\RequirePackage{fix-cm}
\documentclass[smallcondensed]{svjour3}     
\smartqed  
 \usepackage{mathptmx}      
\usepackage{amsmath} 

\usepackage{amsthm}
\usepackage{amssymb}  
\usepackage{enumerate}
\usepackage{cite}
\usepackage{color}
\usepackage{graphicx}
\usepackage[misc]{ifsym}

\newtheorem{lem}{Lemma}
\newcommand{\FL}{$\mathcal{F}^{1,1}_L$}
\newcommand{\Smu}{$\mathcal{S}^{1,1}_{\mu,L}$}

\begin{document}

\title{Global convergence of the Heavy-ball method for convex optimization
\thanks{ This work was sponsored in part by the Swedish
Foundation for Strategic Research (SSF) and the Swedish Research Council
(VR).}
}


\author{Euhanna Ghadimi \and {Hamid Reza Feyzmahdavian} \and Mikael Johansson}


\institute{E. Ghadimi \and  H. R. Feyzmahdavian \and M. Johansson \at
              Department of Automatic Control, School of Electrical Engineering and ACCESS Linnaeus
Center, Royal Institute of Technology - KTH, Stockholm, Sweden. \\
              \email{euhanna@kth.se}
     \and
         H. R. Feyzmahdavian \\
           \email{hamidrez@kth.se}
           \\
           M. Johansson \\
           \email{mikaelj@kth.se}
}

\date{Received: date / Accepted: date}

\maketitle

\begin{abstract}
This paper establishes global convergence and provides global bounds of the convergence rate of the Heavy-ball method for convex optimization problems. When the objective function has Lipschitz-continuous gradient, we show that the Ces{\'a}ro average of the iterates converges to the optimum at a rate of $\mathcal{O}(1/k)$ where $k$ is the number of iterations. When the objective function is also strongly convex, we prove that the Heavy-ball iterates converge linearly to the unique optimum. 
\end{abstract}
\section{Introduction}
\label{intro}
First-order convex optimization methods have a rich history dating back to 1950's~\cite{Hestenes:52,Frank:56,KjA:58}. 
Recently, these methods have attracted significant interest, both in terms of new theory~\cite{Devolder:13,GuN:14,DrT:14} and in terms of applications in numerous areas such as signal processing~\cite{BeT:09}, machine learning~\cite{Xia:14} and control~\cite{NeN:14}. One reason for this renewed interest is that first-order methods
have a small per-iteration cost and are attractive in large-scale and distributed settings. But the development has also been fuelled by the development of accelerated methods with optimal convergence rates~\cite{Nes:83} and re-discovery of methods that are not only order-optimal, but also have optimal convergence times for smooth convex problems~\cite{Pol:64}. In spite of all this progress, some very basic questions about the achievable convergence speed of first-order convex optimization methods are still open~\cite{DrT:14}.

The basic first-order method is the gradient descent algorithm. For unconstrained convex optimization problems with objective functions that have Lipschitz-continuous gradient, the method produces iterates that are guaranteed to converge to the optimum at the rate $\mathcal{O}(1/k)$ where $k$ is the number of iterations. When the objective function is also strongly convex, the iterates are guaranteed to converge at a linear rate~\cite{Pol:87}.

In the early 1980's, Nemirovski and Yudin~\cite{NeY:83} proved that no first-order method can converge at a rate faster than $\mathcal{O}(1/k^2)$ on convex optimization problems with Lipschitz-continuous gradient. This created a gap between the guaranteed convergence rate of the gradient method and what could potentially be achieved. This gap was closed by Nesterov, who presented an accelerated first-order method that converges as ${\mathcal O}(1/k^2)$~\cite{Nes:83}. Later, the method was generalized to also attain linear convergence rate for strongly convex objective functions, resulting in the first truly order-optimal first-order method for convex optimization~\cite{Nes:04}. The accelerated first-order methods combine gradient information at the current and the past iterate, as well as the iterates themselves~\cite{Nes:04}. For strongly convex problems, Nesterov's method can be tuned to yield a better convergence factor than the gradient iteration, but it is not known how small the convergence factor can be made.

When the objective function is twice continuously differentiable, strongly convex and has Lipschitz continuous gradient, the Heavy-ball method by Polyak~\cite{Pol:64} has linear convergence rate and better convergence factor than both the gradient and Nesterov's accelerated gradient method. The Heavy-ball method uses previous iterates when computing the next, but in contrast to Nesterov's method it only uses the gradient at the current iterate. Extensions of the Heavy-ball method to constrained and distributed optimization problems have confirmed its performance benefits over the standard gradient-based methods~\cite{GsJ:13,ObP:14a,Wang:14}.

On the other hand, when the objective function is not necessary convex but has Lipschitz continuous gradient, Zavriev et al.~\cite{ZaK:93} provided sufficient conditions for the Heavy-ball trajectories to converge to a stationary point. However, there are virtually no results on the rate of convergence of the Heavy-ball method for convex problems that are not necessarily twice-differentiable. Recently, Lessard et al~\cite{LRP:14} showed by an example that the Heavy-ball method does not necessarily converge on strongly convex (but not twice differentiable) objective functions even if one chooses step-size parameters according to Polyak's original stability criterion. In general, it is not clear whether the Heavy-ball method performs better than Nesterov's method, or even the basic gradient descent when the objective is not twice continuously differentiable.

The aim of this paper is to contribute to a more complete understanding of first-order methods for convex optimization. We provide a global convergence analysis for the Heavy-ball method on convex optimization problems with Lipschitz-continuous gradient, with and without the additional assumption of strong convexity. We show that if the parameters of the Heavy-ball method are chosen within certain ranges, the running average of the iterates converge to the optimal point at the rate $\mathcal{O}(1/k)$ when the objective function has Lipschitz continuous gradient. Moreover, for the same class of problems, we are able to show that the individual iterates themselves converge at rate $\mathcal{O}(1/k)$ if the Heavy-ball method uses (appropriately chosen) time-varying step-sizes. Finally, if the cost function is also strongly convex,  we show that the iterates converge at a linear rate.

 The rest of the paper is organized as follows. Section~\ref{sec:back} reviews first-order convex optimization algorithms. Global convergence proofs for the Heavy-ball method are presented in Section~\ref{sec:FL} for objective functions with Lipschitz continuous gradient and in Section~\ref{sec:Smu} for objective functions that are also strongly convex. Concluding remarks are given in Section~\ref{sec:conclusion}.

\subsection{Notation}

We let $\mathbb{R}$, $\mathbb{N}$, and $\mathbb{N}_0$ denote the set of real numbers, the set of natural numbers, and the set of natural numbers including zero, respectively. The Euclidean norm is denoted by $\| \cdot \|$. 

%
%

\section{Background}
\label{sec:back}

We consider unconstrained convex optimization problems on the form
\begin{equation}
  \label{eq:unconst.problem}
  \underset{x\in \mathbb{R}^n}{\mbox{minimize}} \,\; f(x)
\end{equation}
where $f:\mathbb{R}^n\rightarrow \mathbb{R}$ is a continuously differentiable convex function. We will provide convergence bounds for the Heavy-ball method for all functions in the following classes.

\begin{definition}
  \label{def:FL}
We say that $f:\mathbb{R}^n\rightarrow \mathbb{R}$ belongs to the class \FL,  if it is convex, continuously differentiable, and its gradient is Lipschitz continuous with constant $L$, \emph{i.e.},
\begin{equation*}
   0\leq f(y)-f(x)-\langle \nabla f(x),\,y-x \rangle\leq \dfrac{L}{2}\Vert x - y\Vert^2,
\end{equation*}
holds for all $x,y \in \mathbb{R}^n$. If $f$ is also strongly convex with modulus $\mu> 0$, \emph{i.e.},
\begin{equation*}
 \dfrac{\mu}{2}\Vert x-y \Vert^2\leq f(y)-f(x)-\langle \nabla f(x),\,y-x \rangle, \quad \forall x,y \in \mathbb{R}^n,
\end{equation*}
then, we say that $f$ belongs to \Smu.
\end{definition}

Our baseline first-order method is gradient descent:
\begin{equation}\label{eq:gradient_iters}
  x_{k+1} = x_k - \alpha \nabla f(x_k),
\end{equation}
where $\alpha$ is a positive step-size parameter. Let $x^\star$ be an optimal point of~\eqref{eq:unconst.problem} and $f^\star = f(x^\star)$. If $f\in$ \FL, then $f(x_k)-f^\star$ associated with the sequence $\{x_k\}$  in~\eqref{eq:gradient_iters} converges at rate $\mathcal{O}(1/k)$. On the other hand, if $f\in$ \Smu, then the sequence $\{x_k\}$ generated by the gradient descent method converges linearly, \textit{i.e.}, there exists $q\in[0,1)$ such that
\begin{align*}
\Vert x_k-x^\star\Vert \leq q^k \Vert x_0 -x^\star\Vert,\quad k\in\mathbb{N}_0.
\end{align*}
The scalar $q$ is called the \textit{convergence factor}. The optimal convergence factor for $f\in$ \Smu is $q=(L-\mu)/(L+\mu)$, attained for $\alpha=2/(L+\mu)$ ~\cite{Pol:87}.

The convergence of the gradient iterates can be accelerated by accounting for the history of iterates when computing the ones to come. Methods in which the next iterate depends not only on the current iterate but also on the preceding ones are called \emph{multi-step methods}. The simplest multi-step extension of gradient descent is the Heavy-ball method:
\begin{equation}
  x_{k+1} = x_k-\alpha \nabla f(x_k) + \beta \left( x_k-x_{k-1} \right), \label{eq:hb_iters_constant}
\end{equation}
for constant parameters $\alpha>0$ and $\beta>0$~\cite{Pol:87}. For the class of twice continuously differentiable strongly convex functions with Lipschitz continuous gradient, Polyak used a local analysis to derive optimal step-size parameters and to show that the optimal convergence factor of the Heavy-ball iterates is $(\sqrt{L}-\sqrt{\mu})/(\sqrt{L}+\sqrt{\mu})$. This convergence factor is always smaller than the one associated with the gradient iterates, and significantly so when the Hessian of the objective function is poorly conditioned. Note that this local analysis requires twice differentiability of the objective functions, and is, therefore, not valid for all $f\in$ \FL nor for all $f\in$ \Smu.

In contrast, Nesterov's fast gradient method~\cite{Nes:04} is a first-order method with better convergence guarantees than the basic gradient method for objectives in \FL and \Smu classes.  In it's simplest form, Nesterov's algorithm with constant step-sizes takes the form
\begin{equation}\label{eq:Nesterov_iters}
  \begin{aligned}
    y_{k+1} & = x_k - \dfrac{1}{L}\nabla f(x_k),\\
    x_{k+1} & = y_{k+1}+\beta(y_{k+1}-y_k).
  \end{aligned}
\end{equation}
When $f\in$ \Smu, the iterates produced by~\eqref{eq:Nesterov_iters} with $\beta=(\sqrt{L}-\sqrt{\mu})/(\sqrt{L}+\sqrt{\mu})$ converge linearly towards the optimal point with a convergence factor $1-\sqrt{\mu/L}$. This factor is smaller than that of the gradient, but larger than that of the Heavy-ball method for twice-differentiable cost functions.

%
%

\section{Global analysis of Heavy-ball algorithm for the class \FL}
\label{sec:FL}

In this section, we consider the Heavy-ball iterates~\eqref{eq:hb_iters_constant} for the objective functions $f\in$ \FL. Our
first result shows that the method is indeed guaranteed to converge globally and estimates
the convergence rate of the Ces{\'a}ro averages of the iterates.

%
%

\begin{theorem}
\label{Theorem 1}
Assume that $f\in$\FL and that
\begin{align}\label{eq:Hb:FL:stabaility_criteria}
\beta\in[0,1),\quad \alpha\in\biggl(0,\dfrac{2(1-\beta)}{L}\biggr).
\end{align}
Then, the sequence $\{x_k\}$ generated by Heavy-ball iteration~\eqref{eq:hb_iters_constant} satisfies
\begin{align}\label{eq:hb_FL_bound}
f(\overline{x}_T)-f^{\star} \leq & \left\{
\begin{array}[l]{ll}
\frac{\Vert x_{0}-x^\star\Vert^2}{2(T+1)}\biggl(\frac{L\beta}{1-\beta}+\frac{1-\beta}{\alpha}\biggr),\;\;&\textup{if}\;\;
\alpha\in\bigl(0,\dfrac{1-\beta}{L}\bigr],\\
\frac{\Vert x_{0}-x^\star\Vert^2}{2(T+1)(2(1-\beta)-\alpha L)}\biggl({L\beta}+\frac{(1-\beta)^2}{\alpha}\biggr),\;\;&\textup{if}\;\;
\alpha\in\bigl[\dfrac{1-\beta}{L},\dfrac{2(1-\beta)}{L}\bigr),
\end{array}
\right.
\end{align}
where $\overline{x}_T$ is the Ces{\'a}ro average of the iterates, i.e., 
\begin{align*}
\overline{x}_T = \frac{1}{T+1}\sum_{k=0}^T x_k.
\end{align*}
\end{theorem}

%
%

\begin{proof}
Assume that $\beta\in[0,1)$, and  let
\begin{align}
p_k=\dfrac{\beta}{1-\beta}(x_k-x_{k-1}),\quad k\in\mathbb{N}_0.
\label{def:p}
\end{align}
Then
\begin{align*}
x_{k+1}+p_{k+1}&=\dfrac{1}{1-\beta}x_{k+1}-\dfrac{\beta}{1-\beta}x_k\overset{\eqref{eq:hb_iters_constant}}{=}x_k+p_k-\dfrac{\alpha}{1-\beta} \nabla f(x_k),
\end{align*}
which implies that
\begin{align}
\Vert x_{k+1}+p_{k+1}-x^\star\Vert^2=&  \Vert x_k+p_k-x^\star\Vert^2 -\dfrac{2\alpha}{1-\beta}\langle x_k+p_k-x^\star,\nabla f(x_k)\rangle +\biggl(\dfrac{\alpha}{1-\beta}\biggr)^2 \Vert \nabla f(x_k)\Vert^2\nonumber\\
\overset{\eqref{def:p}}{=}& \Vert x_k+p_k-x^\star\Vert^2 -\dfrac{2\alpha}{1-\beta}\langle x_k-x^\star,\nabla f(x_k)\rangle \nonumber\\
&-\dfrac{2\alpha\beta}{(1-\beta)^2}\langle x_k-x_{k-1},\nabla f(x_k)\rangle +\left(\dfrac{\alpha}{1-\beta}\right)^2 \Vert \nabla f(x_k)\Vert^2.
\label{Proof 1:0}
\end{align}
Since $f\in$ \FL, it follows from~\cite[Theorem 2.1.5]{Nes:04} that
\begin{equation}
\begin{aligned}
 \dfrac{1}{L}\Vert\nabla f(x_k)\Vert^2&\leq \langle x_k - x^\star , \nabla f(x_k) \rangle,\\
 f(x_k)-f^\star+\dfrac{1}{2L}\Vert\nabla f(x_k)\Vert^2&\leq \langle x_k - x^\star , \nabla f(x_k) \rangle,\\
f(x_k)-f(x_{k-1}) &\leq \langle x_k - x_{k-1} , \nabla f(x_k)) \rangle.
\end{aligned}
\label{eq:proof_nes.ineq}
\end{equation}
Substituting the above inequalities into~\eqref{Proof 1:0} yields
\begin{align*}
  \Vert x_{k+1}+p_{k+1}-x^\star\Vert^2\leq&\Vert x_k+p_k-x^\star\Vert^2 -\dfrac{2\alpha(1-\lambda)}{L(1-\beta)}\Vert \nabla f(x_k)\Vert^2 -\dfrac{2\alpha\lambda}{1-\beta}\bigl(f(x_k)-f^\star\bigr)\nonumber\\
  &-\dfrac{\alpha\lambda}{L(1-\beta)} \Vert \nabla f(x_k)\Vert^2 -\dfrac{2\alpha\beta}{(1-\beta)^2}\bigl(f(x_k)-f(x_{k-1})\bigr) \nonumber\\
&+\left(\dfrac{\alpha}{1-\beta}\right)^2 \Vert \nabla f(x_k)\Vert^2,
\end{align*}
where $\lambda\in (0,1]$ is a parameter which we will use to balance the weights between the first two inequities in~\eqref{eq:proof_nes.ineq}.
Collecting the terms in the preceding inequality, we obtain
\begin{align}
\dfrac{2\alpha}{(1-\beta)}\bigl(\lambda+\dfrac{\beta}{1-\beta}\bigr)&\bigl(f(x_k)-f^{\star}\bigr)+\Vert x_{k+1}+p_{k+1}-x^\star\Vert^2\nonumber\\
&\leq \dfrac{2\alpha\beta}{(1-\beta)^2}\bigl(f(x_{k-1})-f^{\star}\bigr)+ \Vert x_k+p_k-x^\star\Vert^2 \nonumber\\
&\hspace{2mm}+\left(\dfrac{\alpha}{1-\beta}\right)\left(\dfrac{\alpha}{1-\beta}-\dfrac{2-\lambda}{L}\right) \Vert \nabla f(x_k)\Vert^2.
\label{Proof 1:1}
\end{align}
Note that when $\alpha\in[0,(2-\lambda)(1-\beta)/L]$, the last term of~\eqref{Proof 1:1} becomes non-positive and, therefore, can be eliminated from the right-hand-side. Summing~\eqref{Proof 1:1} over $k = 0,\ldots,T$ gives
\begin{align*}
\dfrac{2\alpha\lambda}{(1-\beta)}\sum_{k=0}^T \bigl(f(x_k)-f^{\star}\bigr)&+\sum_{k=0}^T\left(\dfrac{2\alpha\beta}{(1-\beta)^2}\bigl(f(x_k)-f^{\star}\bigr)+\Vert x_{k+1}+p_{k+1}-x^\star\Vert^2\right)\\
&\hspace{1cm}\leq \sum_{k=0}^T\left(\dfrac{2\alpha\beta}{(1-\beta)^2}\bigl(f(x_{k-1})-f^{\star}\bigr)+\Vert x_{k}+p_{k}-x^\star\Vert^2\right),\
\end{align*}
which implies that
\begin{align*}
\dfrac{2\alpha\lambda}{(1-\beta)}\sum_{k=0}^T \bigl(f(x_k)-f^{\star}\bigr)\leq& \dfrac{2\alpha\beta}{(1-\beta)^2}\bigl(f(x_0)-f^{\star}\bigr)+\Vert x_{0}-x^\star\Vert^2.
\end{align*}
Note that as $f$ is convex, we have
\begin{align}\label{eq:FL:mean:func:ineq}
(T+1)f(\overline{x}_T)\leq \sum_{k=0}^T f(x_k).
\end{align}
It now follows that
\begin{align}\label{eq:hb_cvx_proof_final_ineq}
&f(\overline{x}_T)-f^{\star}\leq \dfrac{1}{T+1}\biggl(\dfrac{\beta}{\lambda(1-\beta)}(f(x_{0})-f^\star)+\dfrac{1-\beta}{2\alpha\lambda}\Vert x_{0}-x^\star\Vert^2\biggr).
\end{align}
%
Additionally, according to~\cite[Lemma 1.2.3]{Nes:04}, $f(x_0)-f^\star \leq (L/2)\Vert x_0-x^\star\Vert^2$. The proof is completed by replacing this upper bound in~\eqref{eq:hb_cvx_proof_final_ineq} and setting $\lambda = 1$ for $\alpha\in\bigl(0,(1-\beta)/L\bigr]$ and $\lambda= 2- (\alpha L)/(1-\beta)$ for $\alpha\in\bigl[(1-\beta)/L,2(1-\beta)/L\bigr)$.
\end{proof}
A few remarks regarding the results of Theorem~\ref{Theorem 1} are in order:
first, a similar convergence rate can be proved for the minimum function values within $T$ number of Heavy-ball iterates. More precisely, the sequence $\{x_k\}$ generated by~\eqref{eq:hb_iters_constant} satisfies
\begin{align*}
\min_{0\leq k\leq T} f(x_k)-f^{\star} \leq \mathcal{O}\biggl( \dfrac{\Vert x_0 -x^\star\Vert^2}{T}\biggr),
\end{align*}
for all $T\in\mathbb{N}_0$. Second, for any fixed $\bar{\alpha}\in (0,1/L]$, one can verify that the $\beta\in[0,1)$ which minimizes the convergence factor~\eqref{eq:hb_FL_bound} is $\beta^\star = 1-\sqrt{\bar{\alpha} L}$ which yields the convergence factor
\begin{equation}\nonumber
\begin{aligned}
\min_{0\leq k\leq T} f(x_k)-f^{\star} \leq \dfrac{1}{2(T+1)}\biggl(\dfrac{2\sqrt{\bar{\alpha}L}-\bar{\alpha}L}{\bar{\alpha}}\biggr)\Vert x_0-x^\star\Vert^2.
\end{aligned}
\end{equation}
Note that this convergence factor is always smaller than the one for the gradient descent method obtained by setting $\beta=0$ in~\eqref{eq:hb_FL_bound}, i.e.,
\begin{equation}\nonumber
\begin{aligned}
  f(x_T)-f^{\star} \leq \dfrac{1}{2\bar{\alpha}(T+1)}\Vert x_0-x^\star\Vert^2.
\end{aligned}
\end{equation}
Finally, setting $\bar{\alpha}=1/L$ in the preceding upper bounds, we see that the factors coincide and
equal the best convergence factor of the gradient descent method reported in~\cite{BeT:09}.

Next, we show that our analysis can be strengthened when we use (appropriately chosen)
time-varying step-sizes in the Heavy-ball method. In this case, the individual iterates $x_k$ (and not just their running average) converge with rate $\mathcal{O}(1/k)$.

%
%

\begin{theorem}
Assume that $f\in$\FL and that
\begin{align}\label{eq:Hb:FL:sequence}
\beta_k=\frac{k}{k+2},\quad \alpha_k= \frac{\alpha_0}{k+2},\quad k\in\mathbb{N},
\end{align}
where $\alpha_0\in(0,1/L]$. Then, the sequence $\{x_k\}$ generated by Heavy-ball iteration~\eqref{eq:hb_iters_constant} satisfies
\begin{align}\label{eq:hb_FL_bound-2}
f(x_T)-f^{\star} \leq & \frac{\Vert x_{0}-x^\star\Vert^2}{2\alpha_0(T+1)},\quad T\in\mathbb{N}.
\end{align}
\end{theorem}

%
%

\begin{proof}
The proof  is similar to that of Theorem~\ref{Theorem 1}, so we will be somewhat terse. For $k\in\mathbb{N}_0$, let $p_k=k(x_k-x_{k-1})$. It is easy to verify that
\begin{align*}
x_{k+1}+p_{k+1}&=x_k+p_k-\alpha_0 \nabla f(x_k),
\end{align*}
which together with the inequalities in~\eqref{Proof 1:0} implies that
\begin{align*}
2\alpha_0(k+1)\bigl(f(x_k)-f^{\star}\bigr)+\Vert x_{k+1}+p_{k+1}-x^\star\Vert^2&\leq 2\alpha_0k\bigl(f(x_{k-1})-f^{\star}\bigr)+ \Vert x_k+p_k-x^\star\Vert^2.
\end{align*}
Summing this inequality over $k = 0,\ldots,T$ gives
\begin{align*}
2\alpha_0(T+1)\bigl(f(x_T)-f^{\star}\bigr)+\Vert x_{T+1}+p_{T+1}-x^\star\Vert^2&\leq \Vert x_0-x^\star\Vert^2.
\end{align*}
The proof is complete.
\end{proof}

To illustrate our results, we evaluate the gradient method and the two variations of the
Heavy-ball method on a numerical example. In this example, the objective function is the Moreau proximal envelope of the function $f(x)=(1/c)\Vert x\Vert$:

\begin{equation}\label{eq:Moreau_envelop}
  f(x) = \left\{
  \begin{aligned}
    &\dfrac{1}{c}\Vert x\Vert - \dfrac{1}{2c^2} & \Vert x\Vert\geq \dfrac{1}{c},\\
    & \dfrac{1}{2}\Vert x\Vert^2 &  \Vert x\Vert\leq \dfrac{1}{c},
  \end{aligned}
  \right.
\end{equation}
with $c=5$ and $x\in\mathbb{R}^{50}$. One can  verify that $f(x) \in$ \FL, i.e., it is convex and continuously differentiable with Lipschitz constant $L = 1$~\cite{Rockafellar:98}. First-order methods designed to find the minimum of this cost function are expected to pertain very poor convergence behavior~\cite{DrT:14}. For the Heavy-ball algorithm with constant step-sizes~\eqref{eq:hb_iters_constant} we chose $\beta=0.5$ and $\alpha=1/L$, for the variant with time varying step-sizes~\eqref{eq:Hb:FL:sequence} we used $\alpha_0 = 1/L$ whereas the gradient algorithm was implemented with the step-size $\alpha=1/L$. Fig.~\ref{fig:cvx_example} shows the progress of the objective values
towards the optimal solution. The plot suggests that $\mathcal{O}(1/k)$ is a quite accurate convergence rate estimate for the Heavy-ball and the gradient method.
\begin{figure}[thpb]
      \centering
      \includegraphics[scale=0.5]{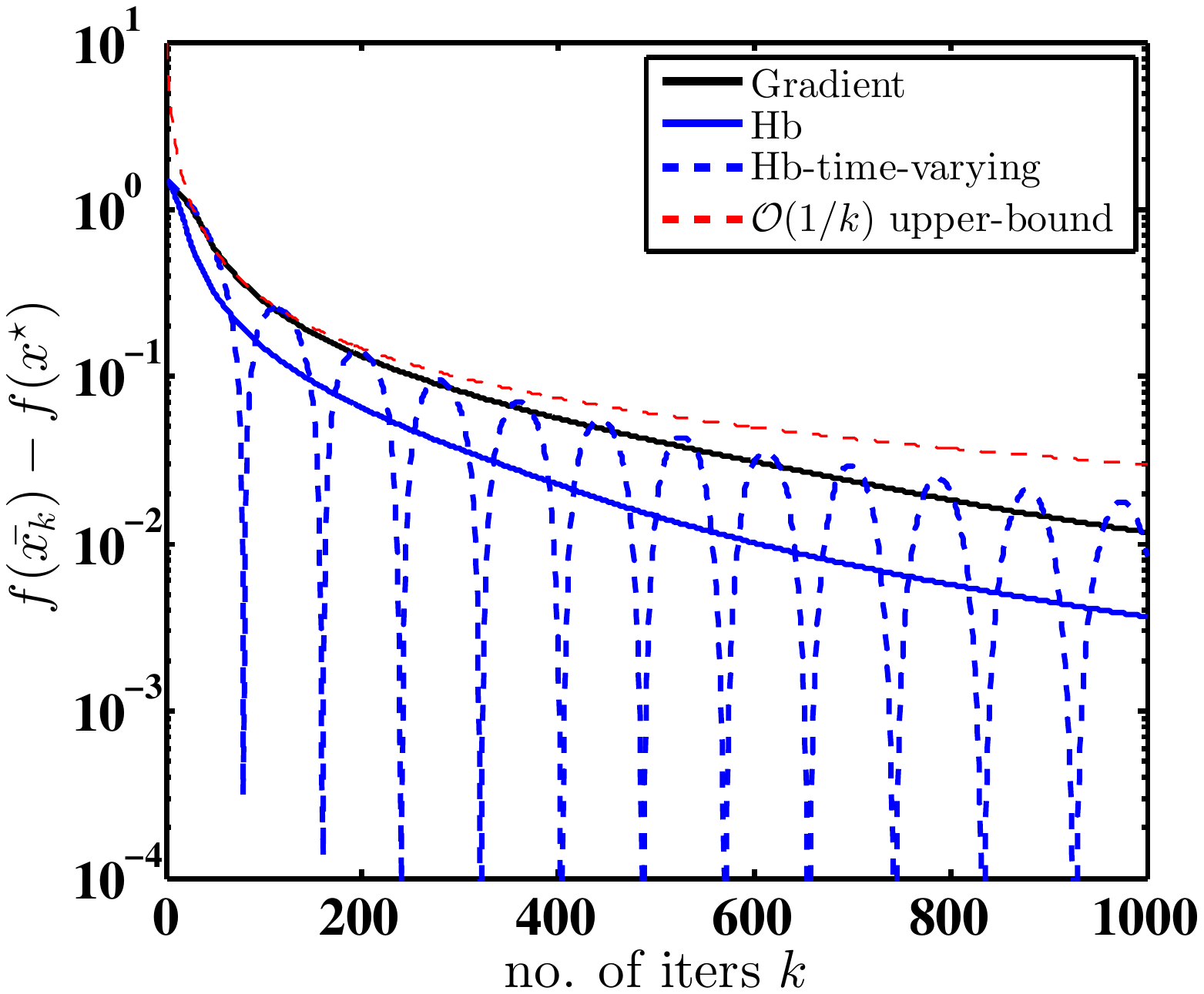}
      \caption{Comparison of the progress of the objective values evaluated at the Ces{\'a}ro average of the iterates of the gradient descent and Heavy-ball methods, and of the primal variable itself for Heavy-ball iterates with time-varying step-sizes. Included for reference is also an ${\mathcal O}(1/k)$ upper bound.}
      \label{fig:cvx_example}
   \end{figure}

%
%
\subsection{Convergence analysis of Nesterov's method with constant step-sizes}

For objective functions $f\in $ \Smu, it is possible to use constant step-sizes in Nesterov's method and still guarantee a linear rate of convergence~\cite{Nes:04}. For the objective functions on the class \FL, however, to the best of our knowledge no convergence result exists for Nesterov's method with fixed step-sizes. Using a similar analysis as in the previous section, we can derive the following convergence rate bound.
%
%

\begin{theorem}\label{thm:FL:Nesterov}
Assume that $f\in$\FL and that $\beta\in [0,1)$. Then the sequence $\{x_k\}$ generated by Nesterov's iteration~\eqref{eq:Nesterov_iters} satisfies
\begin{align}\label{eq:Nesterov_FL_bound}
&f(\overline{x}_T)-f^{\star}\leq \dfrac{1}{T+1}\biggl(\dfrac{\beta}{1-\beta}(f(x_{0})-f^\star)+\dfrac{L(1-\beta)}{2}\Vert x_{0}-x^\star\Vert^2\biggr).
\end{align}
\end{theorem}

%
%

\begin{proof}
Assume that $\beta\in[0,1)$, and  let
\begin{align}
p_k=\dfrac{\beta}{1-\beta}\biggl(x_k-x_{k-1}+\dfrac{1}{L}\nabla f(x_{k-1})\biggr),\quad k\in\mathbb{N}_0.
\label{def:p_n}
\end{align}
Considering~\eqref{eq:Nesterov_iters} and substituting the $y$-th iterates in {the} $x$-th iterates yields
\begin{align*}
x_{k+1}+p_{k+1}&=\dfrac{1}{1-\beta}x_{k+1}+\dfrac{\beta}{1-\beta}(\dfrac{1}{L}\nabla f(x_k)-x_k)\overset{\eqref{eq:Nesterov_iters}}{=}x_k+p_k-\dfrac{1}{L(1-\beta)} \nabla f(x_k),
\end{align*}
which implies that
\begin{equation}
\begin{aligned}
\Vert x_{k+1}+p_{k+1}-x^\star\Vert^2=&  \Vert x_k+p_k-x^\star\Vert^2 -\dfrac{2}{L(1-\beta)}\langle x_k+p_k-x^\star,\nabla f(x_k)\rangle +\dfrac{1}{L^2(1-\beta)^2} \Vert \nabla f(x_k)\Vert^2\\
\overset{\eqref{def:p_n}}{=}& \Vert x_k+p_k-x^\star\Vert^2 -\dfrac{2}{L(1-\beta)}\langle x_k-x^\star,\nabla f(x_k)\rangle -\dfrac{2\beta}{L(1-\beta)^2}\langle x_k-x_{k-1},\nabla f(x_k)\rangle\\
& - \dfrac{2\beta}{L^2 (1-\beta)^2}\langle \nabla f(x_{k-1}),\, \nabla f(x_k)\rangle+\dfrac{1}{L^2(1-\beta)^2} \Vert \nabla f(x_k)\Vert^2\\
&\overset{\eqref{eq:proof_nes.ineq}}{\leq } \Vert x_k+p_k-x^\star\Vert^2 -\dfrac{2}{L(1-\beta)} (f(x_k)-f^\star)-\dfrac{1}{L^2(1-\beta)}\Vert\nabla f(x_k)\Vert^2\\
&-\dfrac{2\beta}{L(1-\beta)^2}(f(x_k)-f(x_{k-1})) -\dfrac{\beta}{L^2(1-\beta)^2}\Vert \nabla f(x_k)-\nabla f(x_{k-1})\Vert^2\\
& - \dfrac{2\beta}{L^2 (1-\beta)^2}\langle \nabla f(x_{k-1}),\, \nabla f(x_k)\rangle+\dfrac{1}{L^2(1-\beta)^2} \Vert \nabla f(x_k)\Vert^2.
\end{aligned}\nonumber
\end{equation}
After rearrangement of terms, we thus have
\begin{equation}
\begin{aligned}
\dfrac{2}{L(1-\beta)^2} (f(x_k)-f^\star)+\Vert x_{k+1}+p_{k+1}-x^\star\Vert^2 &\leq \dfrac{2\beta}{L(1-\beta)^2}(f(x_{k-1})-f^\star)+ \Vert x_k+p_k-x^\star\Vert^2 \\
& -\dfrac{\beta}{L^2(1-\beta)^2}\Vert \nabla f(x_{k-1})\Vert^2
\end{aligned}\label{eq:FL:Nesterov:proof:ineq}
\end{equation}
Multiplying the sides of~\eqref{eq:FL:Nesterov:proof:ineq} in $L/2$ and summing over $k = 0,\ldots,T$ gives
\begin{align*}
\dfrac{1}{1-\beta}\sum_{k=0}^T \bigl(f(x_k)-f^{\star}\bigr)&+\sum_{k=0}^T\left(\dfrac{\beta}{(1-\beta)^2}\bigl(f(x_k)-f^{\star}\bigr)+\dfrac{L}{2}\Vert x_{k+1}+p_{k+1}-x^\star\Vert^2\right)\\
&\hspace{1cm}\leq \sum_{k=0}^T\left(\dfrac{\beta}{(1-\beta)^2}\bigl(f(x_{k-1})-f^{\star}\bigr)+\dfrac{L}{2}\Vert x_{k}+p_{k}-x^\star\Vert^2\right),\
\end{align*}
which implies that
\begin{align*}
\dfrac{1}{1-\beta}\sum_{k=0}^T \bigl(f(x_k)-f^{\star}\bigr)\leq& \dfrac{\beta}{(1-\beta)^2}\bigl(f(x_0)-f^{\star}\bigr)+\dfrac{L}{2}\Vert x_{0}-x^\star\Vert^2.
\end{align*}
Using the convexity inequality~\eqref{eq:FL:mean:func:ineq} concludes the proof.
\end{proof}

Recently, Allen-Zou and Orrechia \cite{ZhO:14} demonstrated that another fast gradient method
due to Nesterov~\cite{Nesterov:13} converges with constant step-sizes for all $f\in$ \FL. That method generates
iterates in the following manner
\begin{equation}\label{eq:Nesterov:mirror}
\begin{aligned}
  y_{k+1} & = x_k - \dfrac{1}{L}\nabla f(x_k),\\
  z_{k+1}& = \underset{z\in\mathbb{R}^n}{\mbox{arg min}} \{V_x(z)+\alpha\langle \nabla f(x_k),\, z-z_k \rangle\},\\
  x_{k+1}& = \tau z_{k+1} + (1-\tau)y_{k+1},
\end{aligned}
\end{equation}
where $\tau\in[0,1]$, and $V_x(\cdot)$ is the Bergman divergence function~\cite{ZhO:14}.Similar to Theorem~\ref{thm:FL:Nesterov}, it has been shown in \cite{ZhO:14} that the Ces{\'a}ro average of the iterates generated by~\eqref{eq:Nesterov:mirror} converges to the optimum at a rate of $\mathcal{O}(1/k)$. Note that while both iterations~\eqref{eq:Nesterov_iters} and~\eqref{eq:Nesterov:mirror} enjoy the same global rate of convergence, the two schemes are remarkably different computationally. In particular,~\eqref{eq:Nesterov:mirror} requires two gradient computations per iteration, as opposed to one gradient computation needed in~\eqref{eq:Nesterov_iters}.
%
%

\section{Global analysis of Heavy-ball algorithm for the class \Smu}
\label{sec:Smu}

In this section, we focus on objective functions in the class \Smu  and derive a  global linear rate of convergence for the Heavy-ball algorithm.  In our convergence analysis, we will use the following simple lemma on convergence of sequences.
%
%
\begin{lem}\label{lem:hybrid_linear_convergece}
 Let $\{A_k\}_{k\geq0}$ and $\{B_k\}_{k\geq0}$ be nonnegative sequences of real numbers satisfying
 \begin{equation}\label{eq:lem:hybrid_ineq}
   A_{k+1}+ b B_{k+1}\leq a_1 A_k + a_2 A_{k-1}+ c B_k,\quad k\in \mathbb{N}_0
 \end{equation}
 with constants $a_1,a_2, b\in \mathbb{R}_+$ and $c \in \mathbb{R} $. Moreover, assume that
 \begin{equation}\nonumber
   A_{-1}=A_0,\quad a_1+a_2<1, \quad c<b.
 \end{equation}
 Then, the sequence $\{A_k\}_{k\geq0}$ generated by~\eqref{eq:lem:hybrid_ineq} satisfies
\begin{equation}\label{eq:lem:linear_rate}
  A_{k}\leq q^k ((q-a_1+1)A_0+ cB_0), 
\end{equation}
where $q\in [0,1)$ is given by
\begin{equation*}
\quad q=\max\left\{ \dfrac{c}{b},\, \dfrac{a_1+\sqrt{a_1^2+4a_2}}{2} \right\}.
\end{equation*}

\begin{proof}
 It is easy to check that \eqref{eq:lem:linear_rate} holds for $k=0$. Let $\gamma\geq 0$. From~\eqref{eq:lem:hybrid_ineq}, we have
\begin{equation}
\begin{aligned}
    A_{t+1} + \gamma A_t + b B_{t+1}&\leq  (a_1+\gamma) A_t + a_2 A_{t-1} + c B_t \\
    & = (a_1+\gamma) (A_t+ \dfrac{a_2}{a_1+\gamma}A_{t-1}+\dfrac{c}{a_1+\gamma}B_t)\\
    &\leq (a_1+\gamma) (A_t+ \gamma A_{t-1}+ b B_t).
\end{aligned}
\label{Eq:lemma 1-1}
  \end{equation}
Note that the last inequality holds if 
\begin{equation}\label{eq:lem:hybrid_proof_ineq1}
\begin{aligned}
    \dfrac{a_2}{a_1+\gamma}\leq \gamma, \quad \dfrac{c}{a_1+\gamma}\leq b.
\end{aligned}
  \end{equation}
  The first term in \eqref{eq:lem:hybrid_proof_ineq1} along with $\gamma\geq 0$ is equivalent to have $ (-a_1+\sqrt{a_1^2+4a_2})/2\leq \gamma$. Moreover, the second condition in~\eqref{eq:lem:hybrid_proof_ineq1} can be rewritten as $ c/b - a_1\leq \gamma$. Thus,  if 
  \begin{equation}
 \gamma= \max\left\{\dfrac{-a_1+\sqrt{a_1^2+4a_2}}{2}, \dfrac{c}{b} - a_1,0\right\},
  \end{equation}
then \eqref{eq:lem:hybrid_proof_ineq1} holds. Denoting  $q\triangleq a_1+{\gamma}<1$ , it follows from~\eqref{Eq:lemma 1-1} that
\begin{align*}
A_{t+1} + \gamma A_t + b B_{t+1}&\leq  q (A_t + \gamma A_{t-1} + c B_t) \leq \cdots\leq q^{t+1} ((1+\gamma)A_{0}+ c B_{0}).
\end{align*}
Since $A_t $ and $B_{t+1}$ are nonnegative,~\eqref{eq:lem:linear_rate} holds. The proof is complete.

\end{proof}
\end{lem}

We are now ready for the main result in this section.

%
%

\begin{theorem}
\label{thm:Smu}
Assume that $f\in$~\Smu and that
%
\begin{align}\label{eq:Hb:Smu:stability_criteria}
 \alpha\in(0,\dfrac{2}{L}),\quad 0\leq  \beta<\dfrac{1}{2}\biggl( \dfrac{\mu \alpha}{2}+\sqrt{\dfrac{\mu^2\alpha^2}{4}+4(1-\dfrac{\alpha L}{2})} \biggr) .
\end{align}
Then, the Heavy-ball method~\eqref{eq:hb_iters_constant} converges linearly to a unique optimizer $x^\star$. In particular,
\begin{align}\label{eq:Hb:Smu:rate}
f(x_{k})-f^\star \leq q^k (f(x_0)-f^\star),
\end{align}
where $q\in[0,1)$.
\end{theorem}

%
%

\begin{proof}
For the heavy-ball iterates~\eqref{eq:hb_iters_constant}, we have
\begin{equation}\label{eq:hb:Smu:proof:error_identity}
\begin{aligned}
  \Vert x_{k+1}-x_k\Vert^2& = \alpha^2\Vert \nabla f(x_k)\Vert^2 + \beta^2\Vert x_k- x_{k-1}\Vert^2 - 2\alpha \beta \langle \nabla f(x_k),\, x_k-x_{k-1}\rangle.
\end{aligned}
\end{equation}
Moreover, since $f$ belongs to \FL, it follows from~\cite[Theorem 2.1.5]{Nes:04} and~\eqref{eq:hb_iters_constant} that
\begin{equation}\label{eq:hb:Smu:proof:f_upperbound}
  \begin{aligned}
    f(x_{k+1})-f^\star &\leq f(x_k) - f^\star - \alpha(1-\dfrac{\alpha L}{2})\Vert \nabla f(x_k)\Vert^2 + \dfrac{L\beta^2}{2}\Vert x_k-x_{k-1}\Vert^2\\
    & + \beta (1-\alpha L) \langle\nabla f(x_k), \, x_k-x_{k-1}\rangle.
  \end{aligned}
\end{equation}
Let $\theta\in (0,1)$, multiply both sides of~\eqref{eq:hb:Smu:proof:error_identity} by $L\theta/(2-2\theta)$, and add the resulting identity to~\eqref{eq:hb:Smu:proof:f_upperbound} to obtain
\begin{equation}\label{eq:hb:Smu:proof:main_eq1}
\begin{aligned}
  f(x_{k+1})-f^\star + &\dfrac{L\theta}{2(1-\theta)} \Vert x_{k+1}-x_k\Vert^2  \leq f(x_k)-f^\star \\
  &\qquad + \alpha \biggl( \dfrac{L}{2(1-\theta)} \alpha - 1\biggr) \Vert \nabla f(x_k)\Vert^2 + \dfrac{L\beta^2}{2(1-\theta)} \Vert x_k-x_{k-1}\Vert^2\\
  &\qquad + \beta (1- \dfrac{\alpha L}{1-\theta})\langle \nabla f(x_k), x_k-x_{k-1}\rangle.
\end{aligned}
\end{equation}
Assume that $(1-\theta)/L\leq \alpha< 2(1-\theta)/L$. Then, since $f\in$ \Smu, it follows from~\cite[Theorem 2.1.10]{Nes:04} that
\begin{equation}\nonumber
\begin{aligned}
  f(x_{k+1})-f^\star + &\dfrac{L\theta}{2(1-\theta)} \Vert x_{k+1}-x_k\Vert^2  \leq f(x_k)-f^\star \\
  &\qquad+  2\alpha\mu \biggl( \dfrac{L}{2(1-\theta)} \alpha - 1\biggr) (f(x_k)-f^\star) +\dfrac{L\beta^2}{2(1-\theta)} \Vert x_k-x_{k-1}\Vert^2\\
  & \qquad+ \beta(1- \dfrac{\alpha L}{1-\theta}) (f(x_k)-f(x_{k-1}))+\dfrac{\beta\mu}{2}(1- \dfrac{\alpha L}{1-\theta})  \Vert x_k - x_{k-1}\Vert^2.
\end{aligned}
\end{equation}
Collecting terms yields
\begin{equation}\label{eq:hb:Smu:proof:main_eq2}
\begin{aligned}
  f(x_{k+1})-f^\star + &\underset{b}{\underbrace{\dfrac{L\theta}{2(1-\theta)}}} \Vert x_{k+1}-x_k\Vert^2  \leq \underset{a_1}{\underbrace{\biggl(1-2\alpha \mu (1-\dfrac{\alpha L}{2(1-\theta)})-\beta( \dfrac{\alpha L}{1-\theta}-1)\biggr)}} (f(x_k)-f^\star)\\
  &\qquad + \underset{a_2}{\underbrace{\beta( \dfrac{\alpha L}{1-\theta}-1)}} (f(x_{k-1})-f^\star)+ \underset{c}{\underbrace{\dfrac{\beta}{2}\biggl(\mu(1- \dfrac{\alpha L}{1-\theta})+ \dfrac{L\beta}{1-\theta}\biggr)}}  \Vert x_k - x_{k-1}\Vert^2,
\end{aligned}
\end{equation}
which is on the form of Lemma~\ref{lem:hybrid_linear_convergece} if we identify $A_k$ with $f(x_k)-f^{\star}$ and $B_k$ with $\Vert x_k-x^{\star}\Vert^2$. It is easy to verify that for $\theta\in(0,1)$ and $(1-\theta)/L\leq \alpha< 2(1-\theta)/L$, one has
\begin{equation}\nonumber
 b>0,\quad  a_1+a_2 < 1.
\end{equation}
Moreover, provided that
\begin{equation}\nonumber
   0\leq  \beta<\dfrac{1}{2}\biggl( \dfrac{\mu}{L}(\alpha L+\theta-1)+\sqrt{\dfrac{\mu^2}{L^2}(\alpha L+\theta-1)^2+4\theta} \biggr),
\end{equation}
it holds that $c<b$ and consequently one can apply Lemma~\ref{lem:hybrid_linear_convergece} with constants $a_1,a_2,b$, and $c$ to conclude the linear convergence~\eqref{eq:Hb:Smu:rate}. Defining $\lambda \triangleq 1-\theta $ the stability criteria reads
\begin{align*}
\lambda\in(0,1),\quad \dfrac{\lambda}{L}\leq \alpha< \dfrac{2\lambda}{L},\quad 0\leq  \beta<\dfrac{1}{2}\biggl( \dfrac{\mu}{L}(\alpha L -\lambda)+\sqrt{\dfrac{\mu^2}{L^2}(\alpha L-\lambda)^2+4(1-\lambda)} \biggr).
\end{align*}
The first two conditions can be rewritten as
\begin{equation}\nonumber
  \alpha \in (0,\dfrac{2}{L}),\quad \lambda\in\biggl(\dfrac{\alpha L}{2}, \mbox{min}(\alpha L,1)\biggr).
\end{equation}
Substituting $\lambda=\alpha L/2$ into the upper stability bound on $\beta$ completes the proof.
\end{proof}

This result extends earlier theoretical results for $\mathcal{S}^{2,1}_{L,\mu}$ to  \Smu and demonstrates that
the Heavy-ball method has the same rate of convergence as the gradient method and Nesterov's
fast gradient method for this class of objective functions. A few comments regarding our stability criteria~\eqref{eq:Hb:Smu:stability_criteria} are in order.

First, we observe that~\eqref{eq:Hb:Smu:stability_criteria} guarantees stability for a wider range of parameters than the stability criteria~\eqref{eq:Hb:FL:stabaility_criteria} for $f\in$ \FL, and wider ranges of parameters than the stability analysis of the Heavy-ball method for  non-convex cost functions presented in~\cite{ZaK:93}. In particular, when $\alpha$ tends to $2/L$, our stability criterion allows $\beta$ to be as large as $\mu/L$, whereas the stability condition~\eqref{eq:Hb:FL:stabaility_criteria} requires that $\beta$ tends to zero when $\alpha$ reaches $2/L$; see Fig.~\ref{fig:smu_stability}.

\begin{figure}[thpb]
      \centering
      \includegraphics[scale=0.4]{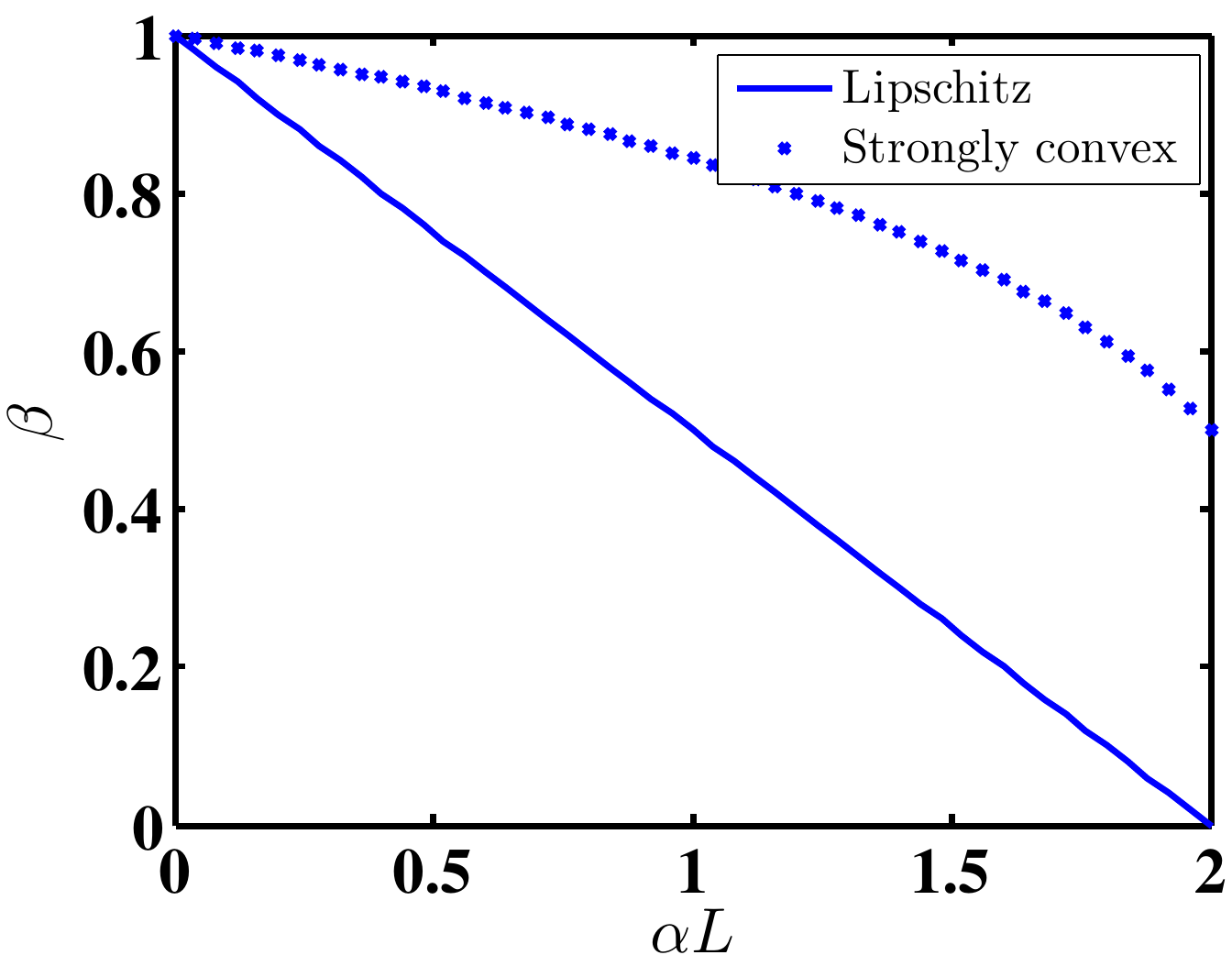}
      \includegraphics[scale=0.4]{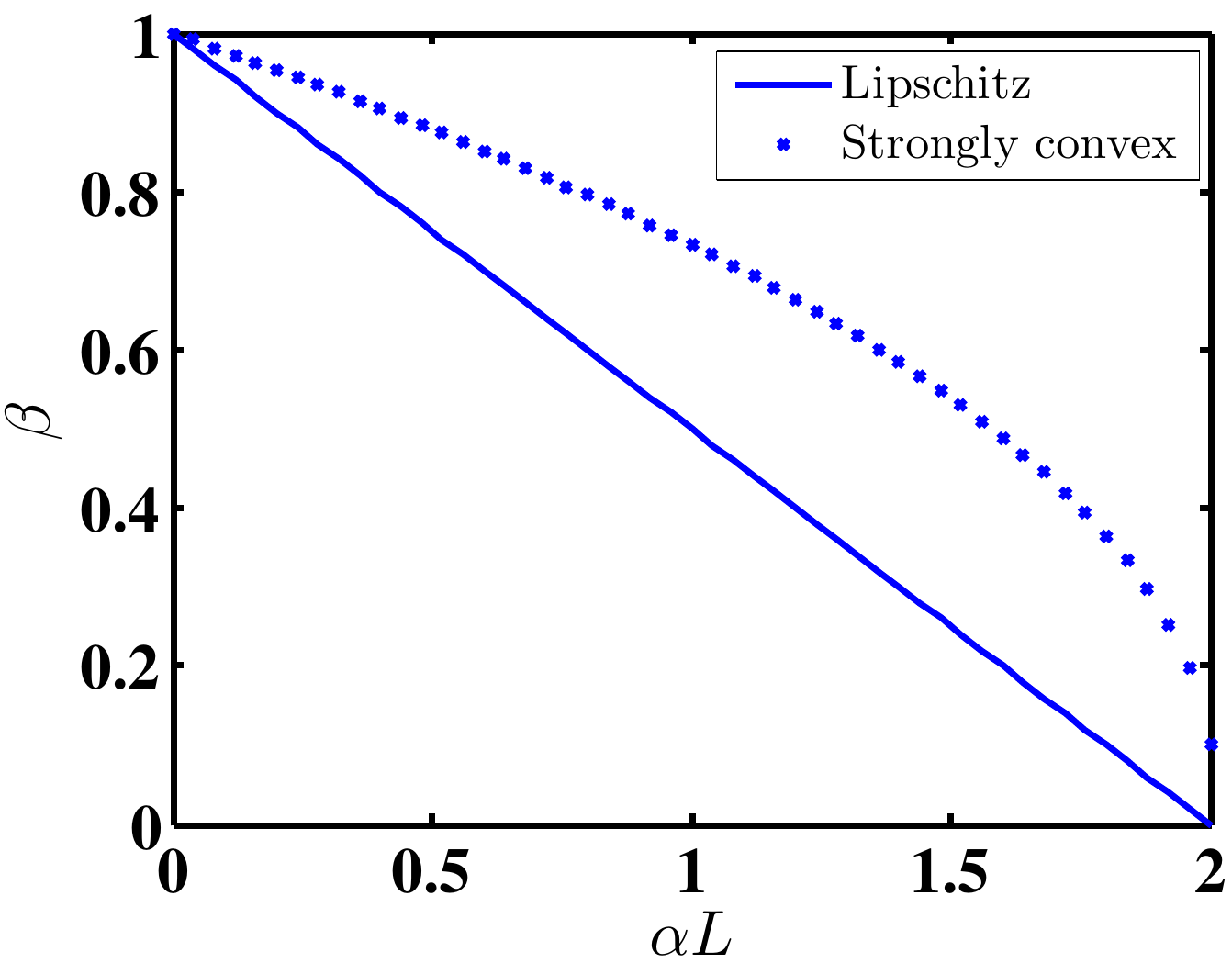}
      \caption{The set of parameters $(\alpha, \beta)$ which guarantee convergence of the Heavy-ball algorithm~\eqref{eq:hb_iters_constant} for  objective functions $f\in$ \FL  (Theorem.~\ref{Theorem 1}) and $f\in$ \Smu  (Theorem.~\ref{thm:Smu}). The left figure uses $L=2,\, \mu=1$ and in the right figure $L=10,\, \mu =1$.} 
      \label{fig:smu_stability}
   \end{figure}

\begin{figure}[thpb]
      \centering
      \includegraphics[scale=0.44]{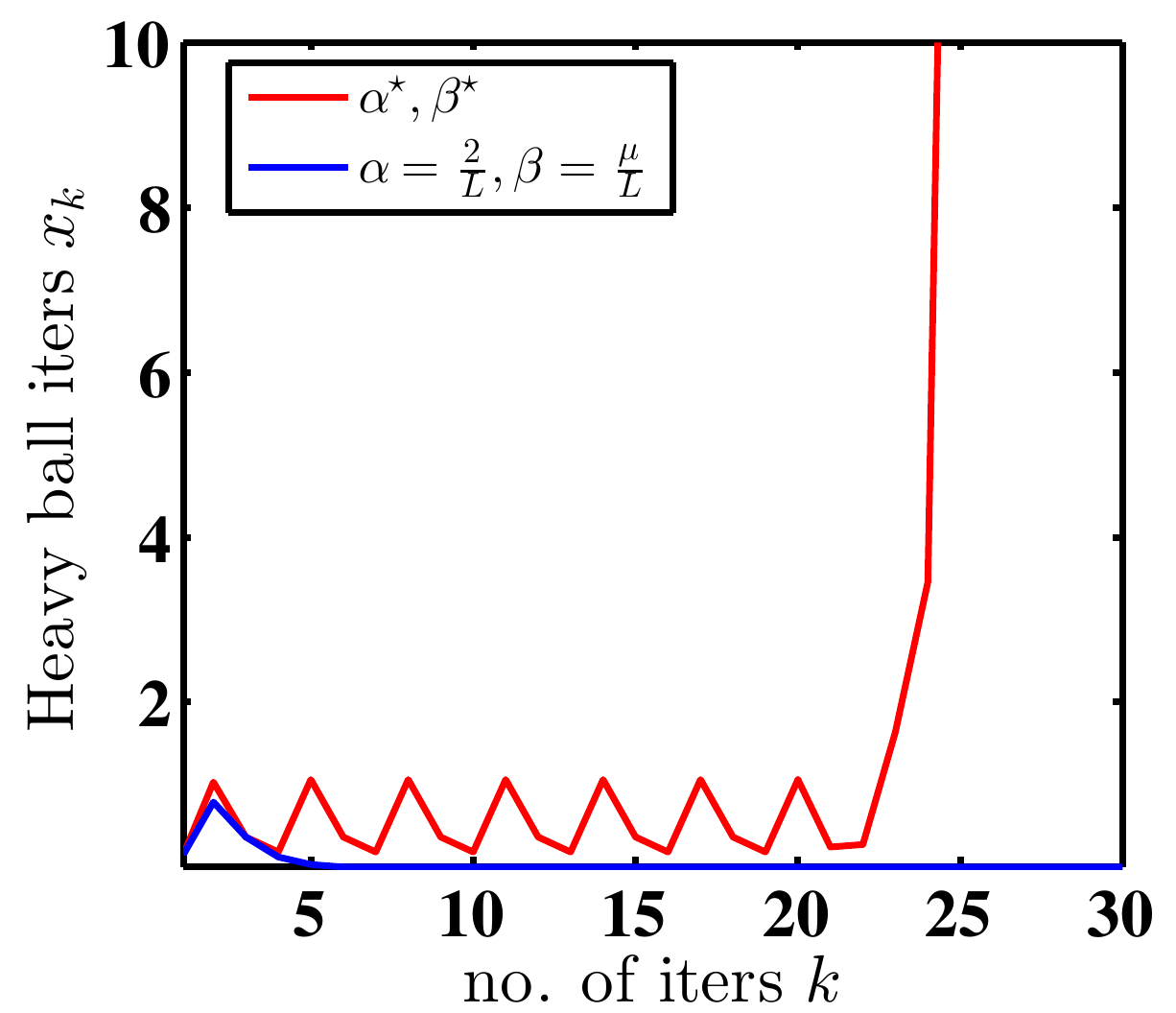}
      \includegraphics[scale=0.44]{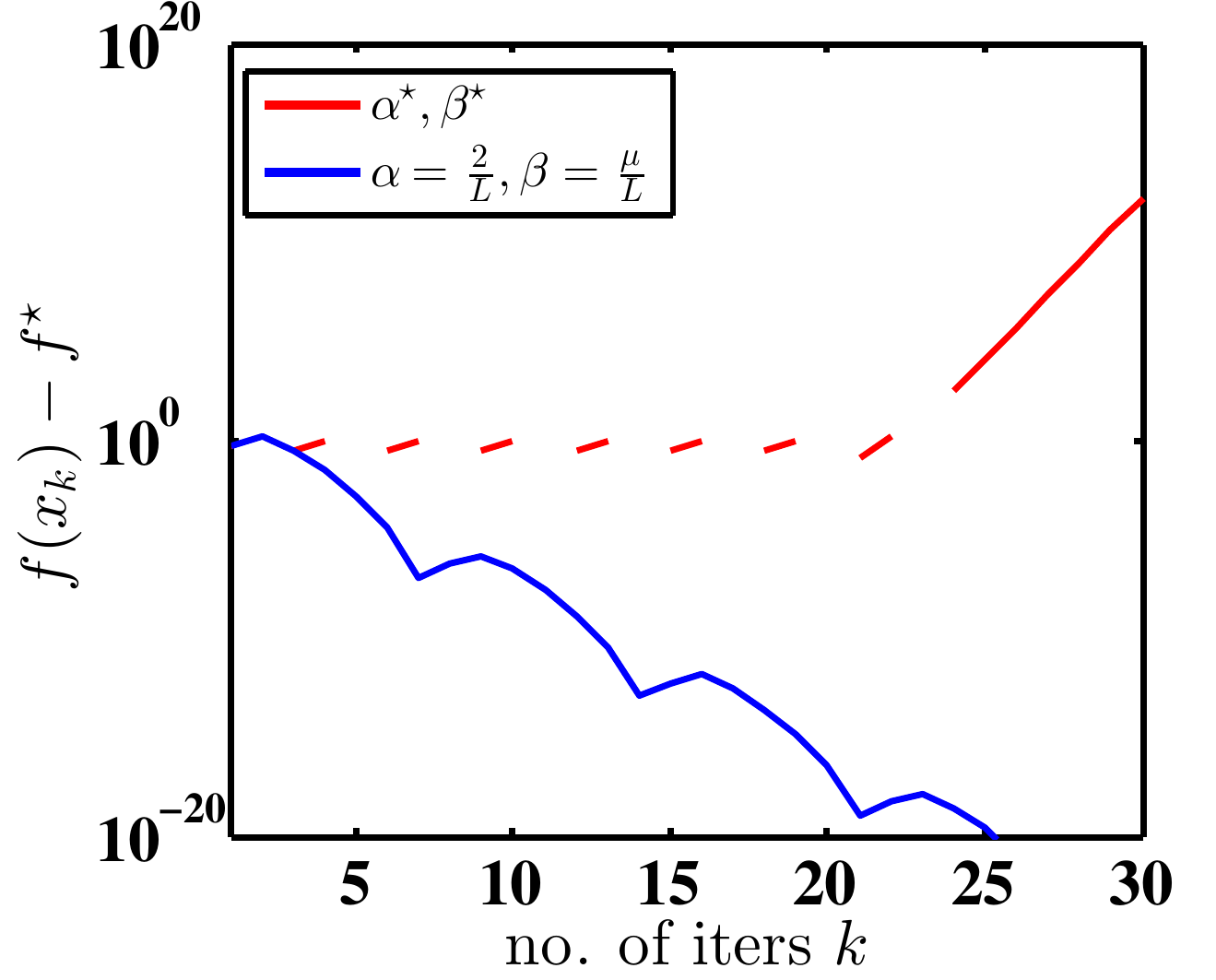}
      \caption{Heavy-ball iterates with optimal step-sizes for $f\in {\mathcal S}_{\mu, L}^{2,1}$ do not converge for the example in~\eqref{eq:lessard_example}. However, parameters that satisfy our new global stability criteria ensure convergence of the iterates.}
      \label{fig:lessard_example}
   \end{figure}
Second, by comparing~\eqref{eq:Hb:Smu:stability_criteria} with $\alpha$ and $\beta$ that guarantee stability for twice differentiable strongly convex functions~\cite{Pol:87}:
\begin{align}
\beta\in[0,1),\quad \alpha\in\biggl(0,\dfrac{2(1+\beta)}{L}\biggr),
\label{HB:pol}
\end{align}
our stability criteria may appear restrictive at first. However, motivated by~\cite{LRP:14}, we consider a counter example where the original stability criteria~\eqref{HB:pol} for twice differentiable strongly convex functions do not hold for the class~\Smu. In particular, let us consider
\begin{equation}\label{eq:lessard_example}
  \nabla f(x) = \left\{
  \begin{array}{ll}
   50 x+45 & \hspace{8.5mm}x<-1,\\
    5x & -1\leq x<0,\\
    50x & \hspace{2mm}0\leq x.
  \end{array}
   \right.
\end{equation}
It is easy to check that $\nabla f$ is continuous and $f\in$ \Smu with $\mu= 5$ and $L = 50$. According to our numerical tests, for initial conditions in the interval $x_0<-0.8$ or $x_0>0.15$, the Heavy-ball method with parameters $\alpha^\star = 4/(\sqrt{L}+\sqrt{\mu})^2$ and $\beta^\star=(\sqrt{L}-\sqrt{\mu})^2/(\sqrt{L}+\sqrt{\mu})^2$ (the optimal step-sizes for the class $\mathcal{S}^{2,1}_{L,\mu}$ in~\cite{Pol:87}) produces non-converging sequences. However, Fig.~\ref{fig:lessard_example} shows that using the maximum value of
$\alpha$ permitted by our global analysis results in iterates that converge to the optimum.

Finally, note that Lemma~\ref{lem:hybrid_linear_convergece} also provides an estimate of the convergence factor of the iterates. In particular, after a few simplifications one can find that when
\begin{equation*}
\alpha\in(0,\frac{1}{L}], \quad \beta=\sqrt{(1-\alpha\mu)(1-\alpha L)},
\end{equation*}
 and $\theta=1-\alpha L$ in~\eqref{eq:hb:Smu:proof:main_eq2}, the convergence factor of the Heavy-ball method~\eqref{eq:hb_iters_constant} is given by $q=1-\alpha \mu$. Note that this factor coincides with the best known convergence factor for the gradient method on \Smu~\cite[Theorem 2, Chapter 1]{Pol:87}. However, supported by the numerical simulations we envisage that the convergence factor could be strengthened even further. This is indeed left as a future work.
%
%

%
%

\section{CONCLUSIONS}
\label{sec:conclusion}

Global stability of the Heavy-ball method has been established for two important classes of convex optimization problems. Specifically, we have shown that when the objective function is convex and has a Lipschitz-continuous gradient, then the Ces{\'a}ro-averages of the iterates converge to the optimum at a rate no slower than $\mathcal{O}(1/k)$, where $k$ is the number of iterations. When the objective function is also strongly convex, we established that the Heavy-ball iterates converge linearly to the unique optimum. 

In our future work, we hope to extend the present results to the constrained optimization problems and derive sharper bounds on the guaranteed convergence factor when $f\in$ \Smu.




\bibliographystyle{IEEEtran}
\bibliography{bibfile}

\end{document}